\pdfoutput=1
\documentclass[12pt,reqno]{amsart}
\usepackage{amsmath,amssymb,amsfonts}
\usepackage{eucal}
\usepackage{graphicx,psfrag}
\usepackage{hyperref}
\usepackage{verbatim} 

\newtheorem{lem}{Lemma}[section]
\newtheorem{cor}[lem]{Corollary}

\newtheorem{thm}[lem]{Theorem}

\numberwithin{equation}{section}

\newcommand\tr{\mathrm{tr}}
\newcommand\bip{\mathrm{bip}}
\newcommand\rank{\operatorname{rank}}

\begin{document}

\title{New Bounds for the Laplacian Spectral Radius of a Signed Graph}

\author{Nathan Reff}
\address{Department of Mathematical Sciences\\Binghamton University (SUNY)\\ Binghamton, NY 13902-6000, U.S.A.}
\email{reff@math.binghamton.edu}

\subjclass[2000]{Primary 05C50; Secondary 05C22}

\keywords{Signed graph, Laplacian spectral radius, Laplacian graph eigenvalues, Kirchhoff matrix, signless Laplacian}

\begin{abstract}
We obtain new bounds for the Laplacian spectral radius of a signed graph.  Most of these new bounds have a dependence on edge sign, unlike previously known bounds, which only depend on the underlying structure of the graph.  We then use some of these bounds to obtain new bounds for the Laplacian and signless Laplacian spectral radius of an unsigned graph by signing the edges all positive and all negative, respectively.
\end{abstract}

\date{\today}
\maketitle

\section{Introduction}

The study of eigenvalues of matrices associated to graphs has a long and rich history.  Researchers have extensively studied the adjacency, Laplacian, normalized Laplacian and recently, signless Laplacian matrices of a graph.

A \emph{signed graph} is a graph with the additional structure that edges are given a sign of either $+1$ or $-1$.  An unsigned graph can be considered a signed graph with all edges signed $+1$.  
Formally, a signed graph is a pair $\Sigma=(\Gamma,\sigma)$ consisting of an \emph{underlying graph} $\Gamma=(V,E)$ and a \emph{signature} $\sigma:E\rightarrow \{+1,-1 \}$.  We define the adjacency matrix $A(\Sigma)=(a_{ij})_{n\times n}$ as
\begin{equation*}
a_{ij}=
\begin{cases} \sigma(e_{ij}) & \text{if }v_i\text{ is adjacent to }v_j,
\\
0 &\text{otherwise.}
\end{cases}
\end{equation*}
The \emph{Laplacian matrix} $L(\Sigma)$ (also called the \emph{Kirchhoff matrix} or \emph{admittance matrix}) of a signed graph is defined as $D(\Sigma)-A(\Sigma)$ where $D(\Sigma)$ is the diagonal matrix of the degrees of vertices of $\Sigma$.  The eigenvalues of $L(\Sigma)$ are real and nonnegative because $L(\Sigma)$ is real and positive semidefinite.

Recently there have been studies of eigenvalue bounds for the Laplacian matrix of a signed graph \cite{MR1950410, MR2156977, MR2477235}.

In this paper we find new bounds for the largest eigenvalue of the Laplacian matrix of a signed graph.  While previous work has generalized many known bounds of unsigned graphs to signed graphs, there are no bounds which depend on edge signs.  One objective of this paper is to begin filling this void.  

A consequence of studying eigenvalue bounds for a signed graph is that we can get information on an unsigned graph's associated matrices.  We obtain bounds for the \emph{Laplacian} eigenvalues by signing all edges of $\Gamma$ positive ($+1$), and likewise we obtain bounds for the \emph{signless Laplacian} (\emph{quasi-Laplacian} or \emph{co-Laplacian}) eigenvalues by signing all edges of $\Gamma$ negative ($-1$).  We get such bounds because the Laplacian matrix of $\Gamma$ is
\begin{equation*}
L(\Gamma)=L(\Gamma,+1)=D(\Gamma)-A(\Gamma)
\end{equation*}
and the signless Laplacian is
\begin{equation*}
Q(\Gamma)=L(\Gamma,-1)=D(\Gamma)+A(\Gamma).
\end{equation*}

The paper is organized in the following manner.  In Section \ref{BackS} we provide definitions and the main lemmas that will be applied in Section \ref{MainR}.  We also mention a method of obtaining upper bounds for the Laplacian spectral radius of a signed graph using the signless Laplacian spectral radius of an unsigned graph.  We include two new bounds for the Laplacian spectral radius of a signed graph as examples.  In Section \ref{MainR} we obtain six lower bounds and one upper bound for the Laplacian spectral radius of a signed graph.  All of these new bounds depend in some way on the edge signs.  In Section \ref{LandSL} we use the bounds from Section \ref{MainR} to get new bounds for the spectral radius of the Laplacian and signless Laplacian of an unsigned graph.  In Section \ref{eg} we apply our bounds to some particular signed graphs and compare these to previously known bounds.

\section{Background}\label{BackS}

All graphs considered here are simple.  We define $V:=\{v_1,v_2,\ldots,v_n\}$ and edges in $E$ are denoted by $e_{ij}=v_i v_j$.  We define $n:=|V|$ and $m:=|E|$.  The largest Laplacian eigenvalue $\lambda_{\max}(L(\Sigma))$ is called the \emph{Laplacian spectral radius} of the signed graph $\Sigma$.  

The sign of a cycle in a signed graph is the product of the signs of its edges.  A signed graph is said to be \emph{balanced} if all of its cycles are positive.  We write $b(\Sigma)$ for the number of connected components of $\Sigma$ that are balanced.  We will write $d_j=\mathrm{deg}(v_j)$ for the degree of vertex $j$.   We define the maximum degree as $\Delta:=\max_{j} d_j$.  We use the following notation for sums of powers of the degrees in a graph: $s_p=\sum_{j=1}^n d_j^p$.  The set of vertices adjacent to a given vertex $v$ will be denoted $N(v)$.  The \emph{average degree} is defined as $\overline{d}:=\sum_{j=1}^n d_j/n$.  The \emph{average 2-degree} is defined as $m_j:=\sum_{v_i\in N(v_j)} d_i/d_j$.  We will write $d_j^+$ and $d_j^-$ for the number of positive and negative edges incident to vertex $v_j$, respectively.  This means that $d_j=d_j^+ + d_j^-$.  We define the \emph{net degree} as $d_j^{\pm}=d_j^+-d_j^-$.  We define several different types of degree vectors that will be used as follows: $\mathbf{d}:=(d_1,\ldots,d_n)\in \mathbb{R}^n$, $\mathbf{d}^{(k)}:=(d_1^k,\ldots,d_n^k)\in \mathbb{R}^n$, $\mathbf{d}^+:=(d_1^+,\ldots,d_n^+)\in \mathbb{R}^n$, $\mathbf{d}^-:=(d_1^-,\ldots,d_n^-)\in \mathbb{R}^n$ and $\mathbf{d}^{\pm}:=(d_1^{\pm},\ldots,d_n^{\pm})\in \mathbb{R}^n$.

We denote the all 1's vector by ${\bf j}=(1,\ldots,1)\in \mathbb{R}^n$.  We will use the notation $A^*$ for the conjugate transpose of the matrix $A$.   We will use the Frobenius matrix norm $\|A\|=\sqrt{\tr(A^*A)}$. Let $A$ be a complex $n\times n$ matrix with eigenvalues $\lambda_j$ where $j\in\{1,2,\ldots,n\}$.  The \emph{spectral radius} of $A$ is defined as $\rho(A)=\max_{1\leq j\leq n}|\lambda_j|$.

Huang and Wang \cite{MR2350694} produced several new bounds for the spectral radius of complex matrices involving traces.  The following lemmas are bounds which can be found in their paper.  We will apply them to the Laplacian matrix of a signed graph in Section \ref{MainR}.

\begin{lem}[\cite{MR2350694}]\label{lem2}
Let $A$ be a complex $n\times n$ matrix with $\rank(A)\leq r\leq n$.
Then
\[ \rho(A)\leq \frac{|\tr(A)|}{r} +\left(\frac{r-1}{2r}\right)^{1/2}\left(c(A)- \frac{|\tr(A)|^2}{r} + \left|\tr(A^2)-\frac{\tr(A)^2}{r} \right|\right)^{1/2}, \]
where
\[c(A)=\left[ \left( \|A\|^2- \frac{|\tr(A)|^2}{n}\right)^2 - \frac{1}{2}\|A^* A- AA^*\|^2 \right]^{1/2} +  \frac{|\tr(A)|^2}{n}. \]
\end{lem}
\begin{lem}[\cite{MR2350694}]
Let $A$ be a complex matrix of order $n$ with $\rank(A)\leq r\leq n$.
Then
\begin{align}
\rho(A)&\geq \sqrt{\frac{|\tr(A)^2-\tr(A^2)|}{r(r-1)}}, \quad r\geq 2; \label{l3e1}\\
\rho(A)&\geq \left(\frac{|2\tr(A^3)-3\tr(A^2)\tr(A)+\tr(A)^3|}{r(r-1)(r-2)} \right)^{1/3}, \quad r\geq 3;\label{l3e2}\\
\rho(A)&\geq \left(\frac{|\tr(A)\tr(A^2)-\tr(A^3)|}{r(r-1)} \right)^{1/3}, \quad r\geq 2.\label{l3e3}
\end{align}
\end{lem}

Several of the bounds in this paper involve the rank of the Laplacian matrix.  We will use the following lemma to further simplify our bounds.
\begin{lem}[\cite{MR676405}]\label{lemRANK} Let $\Sigma$ be a signed graph. Then $\rank(L(\Sigma))=n-b(\Sigma)$.
\end{lem}

Suppose $X\subseteq V$, and $\theta:V\rightarrow \{+1,-1\}$.  \emph{Switching $\Sigma$ by $\theta$} means changing $\sigma$ to $\sigma^{\theta}$ defined by
\[ \sigma^{\theta}(e_{ij})=\theta(v_i)\sigma(e_{ij})\theta(v_j). \]
The switched signed graph is denoted by $\Sigma^{\theta}:=(\Gamma,\sigma^{\theta})$.  We say two signed graphs $\Sigma_1$ and $\Sigma_2$ are \emph{switching equivalent}, written $\Sigma_1\sim \Sigma_2$, if $\Gamma_1=\Gamma_2$ and there exists a function $\theta:V(\Sigma_1)\rightarrow \{+1,-1\}$ such that $\Sigma_1^{\theta}=\Sigma_2$.

Hou, Li and Pan \cite{MR1950410} established the following upper bound on the Laplacian spectral radius of a signed graph in terms of the all-negative graph $(\Gamma,-1)$.

\begin{lem}[\cite{MR1950410}]\label{SGUBNeg} Let $\Sigma=(\Gamma,\sigma)$ be a connected signed graph.  Then
\[ \lambda_{\max}(L(\Sigma)) \leq \lambda_{\max}(L(\Gamma,-1)),\]
with equality if and only if $\Sigma \sim (\Gamma,-1)$.
\end{lem}

Lemma \ref{SGUBNeg} produces upper bounds for the Laplacian spectral radius of a signed graph from known bounds for the signless Laplacian of an unsigned graph.  For example, the following two lemmas generalize to signed graphs the bounds obtained by Wang in \cite{MR2390486} for unsigned graphs.  Wang stated bounds for the Laplacian spectral radius of an unsigned graph, but the same proofs are also valid for the signless Laplacian spectral radius. 
\begin{lem}\label{WThm1}
Let $\Sigma$ be a connected signed graph.  Then
\begin{equation}\label{UB1}
\lambda_{\max}(L(\Sigma))\leq 2 + \max_{e_{ij}\in E(\Sigma)} \sqrt{\frac{(d_i+d_j-2)(d_i^2 \cdot m_i+d_j^2\cdot m_j-2d_id_j)}{d_id_j}}.
\end{equation}
with equality if and only if $\Sigma \sim(\Gamma,-1)$.
\end{lem}

Let $d_{{ij}}=d_i+d_j-2$, where $e_{ij}\in E(\Sigma)$.  Let $\delta_{\Lambda}=\min\{d_{{ij}}:e_{ij}\in E(\Sigma)\}$ and $\Delta_{\Lambda}=\max\{d_{{ij}}:e_{ij}\in E(\Sigma) \}$. 
\begin{lem}\label{WThm2}
Let $\Sigma$ be a connected signed graph of order $n>2$ and size $m$.  Then
\begin{equation}\label{UB2}
\lambda_{\max}(L(\Sigma))\leq 2 + \sqrt{\sum_{v\in V(\Sigma)}d_v^2-2m-(m-1)\delta_{\Lambda}+(\delta_{\Lambda}-1)\Delta_{\Lambda}}.
\end{equation}
with equality if and only if $\Sigma \sim(\Gamma,-1)$.
\end{lem}

Let $\Sigma^+$ and $\Sigma^-$ denote the induced signed subgraphs by all positive edges and all negative edges, respectively.  Hou, Li and Pan \cite{MR1950410} also mention the following relationship, which is established using an interlacing argument.

\begin{lem}[\cite{MR2350694}]\label{HInterlaceLem}
Let $\Sigma$ be a signed graph.  Then
\begin{equation}\label{ILACEB}
\max\{\lambda_{\max}(L(\Sigma^+)),\lambda_{\max}(L(\Sigma^-))\} \leq \lambda_{\max}(L(\Sigma)).
\end{equation}
\end{lem}

The next lemma is a part of the Rayleigh-Ritz Theorem.  It is commonly used to establish lower bounds for the largest eigenvalue of a symmetric matrix.  
\begin{lem}\cite[Theorem 4.2.2 p.176]{MR1084815}\label{RRLem} Let $A\in \mathbb{R}^{n\times n}$ be symmetric. Then
\begin{equation}
 \lambda_{\max}(A)=\max_{\mathbf{x}\in \mathbb{R}^n\backslash \{\mathbf{0}\}} \frac{\mathbf{x}^{\text{T}}A\mathbf{x}}{\mathbf{x}^{\text{T}}\mathbf{x}}.
\end{equation}
\end{lem}

\section{Main Results}\label{MainR}

The following three lower bounds for the Laplacian spectral radius of a signed graph depend on the edge signs.
\begin{thm}\label{thm1} Let $\Sigma$ be a connected signed graph of order $n$.  Then
\begin{equation}\label{NeqLB1}
\frac{2}{n}\cdot \sum_{j=1}^n d_j^- \leq \lambda_{\max}(L(\Sigma)),
\end{equation}
\begin{equation}\label{NeqLB2}
\sqrt{\frac{4}{n}\cdot \sum_{j=1}^n (d_j^-)^2} \leq \lambda_{\max}(L(\Sigma)),
\end{equation}
and
\begin{equation}\label{NeqLB3}
\frac{1}{n^{1/3}}\cdot \Big(4\sum_{j=1}^n d_j(d_j^-)^2-8\sum_{e_{ij}\in E(\Sigma)} \sigma(e_{ij})d_i^-d_j^-\Big)^{1/3} \leq \lambda_{\max}(L(\Sigma)).
\end{equation}
\end{thm}
\begin{proof}
For brevity we write $D$ for $D(\Sigma)$ and $A$ for $A(\Sigma)$.  Let $N_k={\bf j}^{\text{T}} L(\Sigma)^k {\bf j}$.  From Lemma \ref{RRLem} the following lower bound is clear: 
\[ \mathbf{j}^{\text{T}}L(\Sigma)^k\mathbf{j}/\mathbf{j}^{\text{T}}\mathbf{j} \leq (\lambda_{\max}(L(\Sigma)))^k.  \]
We will compute $N_1$, $N_2$ and $N_3$; thus, making inequalities \eqref{NeqLB1}, \eqref{NeqLB2} and \eqref{NeqLB3} true.
We will make use of the following equation when computing each $N_k$:
\begin{equation*}
A{\bf j}= \Big(\sum_{j=1}^n \sigma(e_{1j}),\ldots,\sum_{j=1}^n \sigma(e_{nj})\Big)=(d_1^{\pm},\ldots, d_n^{\pm} ) = \mathbf{d}^{\pm}.
\end{equation*}

Now we compute $N_1$.
\begin{align*}
N_1 &= {\bf j}^{\text{T}} L(\Sigma) {\bf j}={\bf j}^{\text{T}} (D-A) {\bf j} = {\bf j}^{\text{T}}(\mathbf{d}-\mathbf{d}^{\pm}) = 2\sum_{j=1}^n d_j^{-}.
\end{align*}

Similarly, we compute $N_2$. 
\begin{align*}
N_2 &= {\bf j}^{\text{T}} L(\Sigma)^2 {\bf j}= {\bf j}^{\text{T}} (D^2-AD-DA+A^2) {\bf j} \\
&= \mathbf{d}^{\text{T}}\mathbf{d}-(\mathbf{d}^{\pm})^{\text{T}}\mathbf{d}-\mathbf{d}^{\text{T}}\mathbf{d}^{\pm}+(\mathbf{d}^{\pm})^{\text{T}}\mathbf{d}^{\pm} \\
&=\sum_{j=1}^n d_j^2-2\sum_{j=1}^n (d_j^{\pm})d_j + \sum_{j=1}^n (d_j^{\pm})^2  \\
&=4 \sum_{j=1}^n (d_j^-)^2.
\end{align*}

Finally, we compute $N_3$.

\begin{align*}
N_3 &= {\bf j}^{\text{T}} L(\Sigma)^3 {\bf j}= {\bf j}^{\text{T}}(D^3-AD^2-DAD+A^2D-D^2A+ADA+DA^2-A^3){\bf j}  \\
&= {\bf j}^{\text{T}}\mathbf{d}^{(3)} - (\mathbf{d}^{\pm})^{\text{T}} \mathbf{d}^{(2)} - \mathbf{d}^{\text{T}}A\mathbf{d} + (\mathbf{d}^{\pm})^{\text{T}} A \mathbf{d} -(\mathbf{d}^{(2)})^{\text{T}} \mathbf{d}^{\pm}+ (\mathbf{d}^{\pm})^{\text{T}} D\mathbf{d}^{\pm} + \mathbf{d}^{\text{T}}A\mathbf{d}^{\pm} -(\mathbf{d}^{\pm})^{\text{T}} A \mathbf{d}^{\pm}\\
&= \sum_{j=1}^n d_j^3-2\sum_{j=1}^n d_j^{\pm}d_{j}^2 - 2\sum_{e_{ij}\in E(\Sigma)} \sigma(e_{ij})d_i d_j + \sum_{j=1}^n (d_j^{\pm})^2d_{j} \\
& \qquad  \qquad+ 2\sum_{e_{ij}\in E(\Sigma)} \sigma(e_{ij})((d_j^{\pm})d_i+(d_i^{\pm})d_j)- 2\sum_{e_{ij}\in E(\Sigma)} \sigma(e_{ij})(d_i^{\pm})(d_j^{\pm})\\
&= 4\sum_{j=1}^n d_j(d_j^-)^2-8\sum_{e_{ij}\in E(\Sigma)} \sigma(e_{ij})d_i^-d_j^-.\qedhere
\end{align*}
\end{proof}
Notice that inequality \eqref{NeqLB3} may be written as
\[ \frac{1}{n^{1/3}}\cdot \Big(4\sum_{j=1}^n d_j(d_j^-)^2+8\sum_{e_{ij}\in E(\Sigma^-)} d_i^-d_j^- -8\sum_{e_{ij}\in E(\Sigma^+)} d_i^-d_j^-\Big)^{1/3} \leq \lambda_{\max}(L(\Sigma)).\]
It is now clear that the lower bounds \eqref{NeqLB1}, \eqref{NeqLB2} and \eqref{NeqLB3} are actually lower bounds on $\lambda_{\max}(L(\Sigma^-))$.  This fact, and the relationship from inequality \eqref{ILACEB}, suggests that $\Sigma^-$ is an induced signed subgraph that is worth studying.

The following upper bound on the Laplacian spectral radius of a signed graph has a dependence on the number of balanced components.
\begin{thm}\label{UBSGLap} Let $\Sigma$ be a signed graph with at least one edge.
Then
\begin{equation}\label{NewUBSG}
\lambda_{\max}(L(\Sigma))\leq \frac{s_1}{n-b(\Sigma)}+\sqrt{(s_1+s_2)-\frac{s_1+s_2+s_1^2}{n-b(\Sigma)}+\left( \frac{s_1}{n-b(\Sigma)} \right)^2}\;.
\end{equation}

\end{thm}
\begin{proof}

It is obvious that 
\begin{equation}\label{trL}
\tr(L(\Sigma))=\sum_{j=1}^n d_j =2m = s_1.
\end{equation}
Again, we write $D$ for $D(\Sigma)$ and $A$ for $A(\Sigma)$.  Since $A$ has a 0 for each diagonal entry $\tr(AD)=0=\tr(DA)$. Therefore, 
\begin{equation}\label{trL2}
\tr(L(\Sigma)^2)=\tr(D^2)+\tr(A^2)=\sum_{j=1}^n d_j^2 + \sum_{j=1}^n d_j =s_2+s_1.
\end{equation}  
Since $L(\Sigma)$ is a positive semidefinite matrix, the Cauchy-Schwarz inequality gives $\tr(L(\Sigma))^2 \leq n\cdot \tr(L(\Sigma)^2)$.  Now by definition $c(L(\Sigma))=\|L(\Sigma)\|^2=\tr(L(\Sigma)^2)=\sum_{j=1}^n d_j +2m$.  Let $r=n-b(\Sigma)$.
Now inequality (\ref{NewUBSG}) follows from Lemmas \ref{lem2} and \ref{lemRANK}.
\end{proof}

We define $\mathcal{T}(\Sigma)$ as the set of all signed triangles of $\Sigma$.  If $T\in \mathcal{T}(\Sigma)$ with vertices $v_i$, $v_j$ and $v_k$ we can write $T=e_{ij}e_{jk}e_{ki}$.  We will write $\sigma(T)=\sigma(e_{ij})\sigma(e_{jk})\sigma(e_{ki})$ to mean the sign of the triangle $T$.  Let $t(\Gamma)$ be the number of triangles in $\Gamma$.  Let $t^+(\Sigma)$ be the number of triangles in $\Sigma$ which have positive sign.  Similarly, let $t^-(\Sigma)$ be the number of triangles in $\Sigma$ which have negative sign.  We define $t^{\pm}(\Sigma):=t^+(\Sigma) - t^-(\Sigma)$.  

The following three lower bounds on the Laplacian spectral radius of a signed graph depend on the number of balanced components.  Also, inequalities \eqref{TrI2} and \eqref{TrI3} depend on signed triangles.
\begin{thm}\label{TrLBSGLap} Let $\Sigma=(\Gamma,\sigma)$ be a signed graph. Then
\begin{equation}\label{TrI1}
\lambda_{\max}(L(\Sigma)) \geq \sqrt{\frac{s_1^2-s_2 -s_1}{(n-b(\Sigma))(n-b(\Sigma)-1)}}, 
\end{equation}
if $b(\Sigma)\leq n-2$;
\begin{equation}\label{TrI2}
\lambda_{\max}(L(\Sigma)) \geq \left(\frac{2s_3+6s_2-3s_2 s_1+s_1^3-3s_1^2-12t^{\pm}(\Sigma)}{(n-b(\Sigma))(n-b(\Sigma)-1)(n-b(\Sigma)-2)} \right)^{1/3}, 
\end{equation}
if $b(\Sigma)\leq n-3$;
\begin{equation}\label{TrI3}
\lambda_{\max}(L(\Sigma)) \geq \left(\frac{s_1^2 - 3 s_2 + s_1 s_2 - s_3 + 6t^{\pm}(\Sigma)}{(n-b(\Sigma))(n-b(\Sigma)-1)} \right)^{1/3},
\end{equation}
if $b(\Sigma)\geq n-2$.
\end{thm}
\begin{proof}Let $r=n-b(\Sigma)$.  Notice that $\tr(L(\Sigma)^3)=\tr(D^3)+3\tr(A^2D)-3\tr(AD^2)-\tr(A^3)$ and $\tr(AD^2)=0$.  Hence,
\begin{align}\label{trL3}
\tr(L(\Sigma)^3)&=\sum_{j=1}^n d_j^3 + 3\sum_{j=1}^n d_j^2 -\sum_{i=1}^n \sum_{j=1}^n \sum_{k=1}^n \sigma(e_{ij})\sigma(e_{jk})\sigma(e_{ki}) \nonumber\\
&=\sum_{j=1}^n d_j^3 + 3\sum_{j=1}^n d_j^2 - 6 \sum_{T\in \mathcal{T}(\Sigma)} \sigma(T) \nonumber \\
&=s_3 + 3s_2 - 6 t^{\pm}(\Sigma).
\end{align}
With equations \eqref{trL}, \eqref{trL2}, \eqref{trL3} and Lemma \ref{lemRANK}, inequalities \eqref{TrI1}, \eqref{TrI2} and \eqref{TrI3}  follow from \eqref{l3e1}, \eqref{l3e2} and \eqref{l3e3}, respectively. 
\end{proof}

\section{The Laplacian and signless Laplacian of an unsigned graph}\label{LandSL}

Unsigned graphs are the special case of signed graphs where all edges are signed $+1$.  Here we use the results from the previous section to obtain bounds for the spectral radius of the Laplacian and signless Laplacian of an unsigned graph.  This is done by signing all edges $+1$ for the Laplacian and $-1$ for the signless Laplacian. 

The information we get from Theorem \ref{thm1} for the Laplacian matrix $L(\Gamma)=L(\Gamma,+1)$ of an unsigned graph is trivial.  The values of $N_1$, $N_2$ and $N_3$ all equal 0 when considering the signed graph $\Sigma=(\Gamma,+1)$, and the lower bound $0\leq \lambda_{\max}(L(\Gamma,+1))$ is obvious since $L(\Gamma,+1)$ is positive semidefinite.

Theorem \ref{thm1} gives the following lower bounds for the signless Laplacian spectral radius of an unsigned graph.

\begin{cor}
Consider a graph $\Gamma=(V,E)$ of order $n$.  Then
\begin{equation}\label{NeqSLB1}
4m/n=2s_1/n=2\overline{d} \leq \lambda_{\max}(L(\Gamma,-1)),
\end{equation}
\begin{equation}\label{NeqSLB2}
\sqrt{4s_2/n} \leq \lambda_{\max}(L(\Gamma,-1)),
\end{equation}
and
\begin{equation}\label{NeqSLB3}
\frac{1}{n^{1/3}}\cdot \Big(4s_3+8\sum_{e_{ij}\in E(\Sigma)} d_i d_j\Big)^{1/3} \leq \lambda_{\max}(L(\Gamma,-1)).
\end{equation}
Furthermore, equality in \eqref{NeqSLB1}, \eqref{NeqSLB2} and \eqref{NeqSLB3} holds if $\Gamma$ is regular. 
\end{cor}

\begin{proof}
All three inequalities are immediate from Theorem \ref{thm1}.  If $\Gamma$ is regular the left side of inequalities \eqref{NeqSLB1}, \eqref{NeqSLB2} and \eqref{NeqSLB3} all simplify to $2\Delta$.  It is well known (see for example, \cite{MR2312332}) that if $\Gamma$ is regular then $\lambda_{\max}(L(\Gamma,-1))=4m/n=2\Delta$.  
\end{proof}
Inequality \eqref{NeqSLB1} is mentioned by Cvetkovi{\'c}, Rowlinson and Simi{\'c} in \cite{MR2401311}.  However, \eqref{NeqSLB2} and  \eqref{NeqSLB3} appear to be new.

Notice that the rank of the Laplacian matrix appears in Theorems \ref{UBSGLap} and \ref{TrLBSGLap}.  We remind the reader of some basic facts about the Laplacian and signless Laplacian rank for unsigned graphs.  If $\Gamma$ is a graph we will write $c(\Gamma)$ to denote the number of components of $\Gamma$ and $c_{\bip}(\Gamma)$ to denote the number of bipartite components of $\Gamma$.  The following lemma is a special case of Lemma \ref{lemRANK}.

\begin{lem}[\cite{MR1130611, MR2312332}]\label{lemRANKsl} Let $\Gamma$ be a graph.  Then
\begin{align*}
 \rank(L(\Gamma,+1))&=n-c(\Gamma),\\
 \rank(L(\Gamma,-1))&=n-c_{\bip}(\Gamma).
\end{align*}
\end{lem}

Applying Theorems \ref{UBSGLap} and \ref{TrLBSGLap} to the Laplacian and signless Laplacian of an unsigned graph results in the following upper bounds.

\begin{cor}
Let $\Gamma=(V,E)$ be a graph with at least one edge.  Then
\begin{align}
\lambda_{\max}(L(\Gamma,+1))&\leq\frac{s_1}{n-c(\Gamma)}+\sqrt{(s_1+s_2)-\frac{s_1+s_2+s_1^2}{n-c(\Gamma)}+\left( \frac{s_1}{n-c(\Gamma)} \right)^2}, \label{NewUBL}\\
\lambda_{\max}(L(\Gamma,-1))&\leq \frac{s_1}{n-c_{\bip}(\Gamma)}+\sqrt{(s_1+s_2)-\frac{s_1+s_2+s_1^2}{n-c_{\bip}(\Gamma)}+\left( \frac{s_1}{n-c_{\bip}(\Gamma)} \right)^2}.\label{NewUBSL}
\end{align}
\end{cor}

The specialization of Theorem \ref{TrLBSGLap} produces the following.
{\allowdisplaybreaks 
\begin{cor} Let $\Gamma=(V,E)$ be a graph. Then
\begin{align}
\lambda_{\max}(L(\Gamma,+1)) &\geq \sqrt{\frac{s_1^2-s_2 -s_1}{(n-c(\Gamma))(n-c(\Gamma)-1)}},\label{TrIL1}
\intertext{if $c(\Gamma)\leq n-2;$}
\lambda_{\max}(L(\Gamma,+1)) &\geq \left(\frac{2s_3+6s_2-3s_2 s_1+s_1^3-3s_1^2-12t(\Gamma)}{(n-c(\Gamma))(n-c(\Gamma)-1)(n-c(\Gamma))-2)} \right)^{1/3}, \label{TrIL2}
\intertext{if $c(\Gamma)\leq n-3;$}
\lambda_{\max}(L(\Gamma,+1)) &\geq \left(\frac{s_1^2 - 3 s_2 + s_1 s_2 - s_3 + 6t(\Gamma)}{(n-c(\Gamma))(n-c(\Gamma)-1)} \right)^{1/3}, \label{TrIL3}
\intertext{if $c(\Gamma)\leq n-2;$}
\lambda_{\max}(L(\Gamma,-1)) &\geq \sqrt{\frac{s_1^2-s_2 -s_1}{(n-c_{\bip}(\Gamma))(n-c_{\bip}(\Gamma)-1)}}, \label{TrISL1}
\intertext{if $c_{\bip}(\Gamma)\leq n-2;$}
\lambda_{\max}(L(\Gamma,-1)) &\geq \left(\frac{2s_3+6s_2-3s_2 s_1+s_1^3-3s_1^2+12t(\Gamma)}{(n-c_{\bip}(\Gamma))(n-c_{\bip}(\Gamma)-1)(n-c_{\bip}(\Gamma))-2)} \right)^{1/3}, \label{TrISL2}
\intertext{if $c_{\bip}(\Gamma)\leq n-3;$}
\lambda_{\max}(L(\Gamma,-1)) &\geq \left(\frac{s_1^2 - 3 s_2 + s_1 s_2 - s_3 - 6t(\Gamma)}{(n-c_{\bip}(\Gamma))(n-c_{\bip}(\Gamma)-1)} \right)^{1/3}\label{TrISL3},\intertext{if $c_{\bip}(\Gamma))\leq n-2.$} \notag
\end{align}
\end{cor}}
\section{Examples}\label{eg}

Here we present some examples of signed graphs and the various bounds obtained in this paper.  For the signed graphs see Figure \ref{ex1}.

\begin{figure}[h!]
    \includegraphics[width=0.65\textwidth]{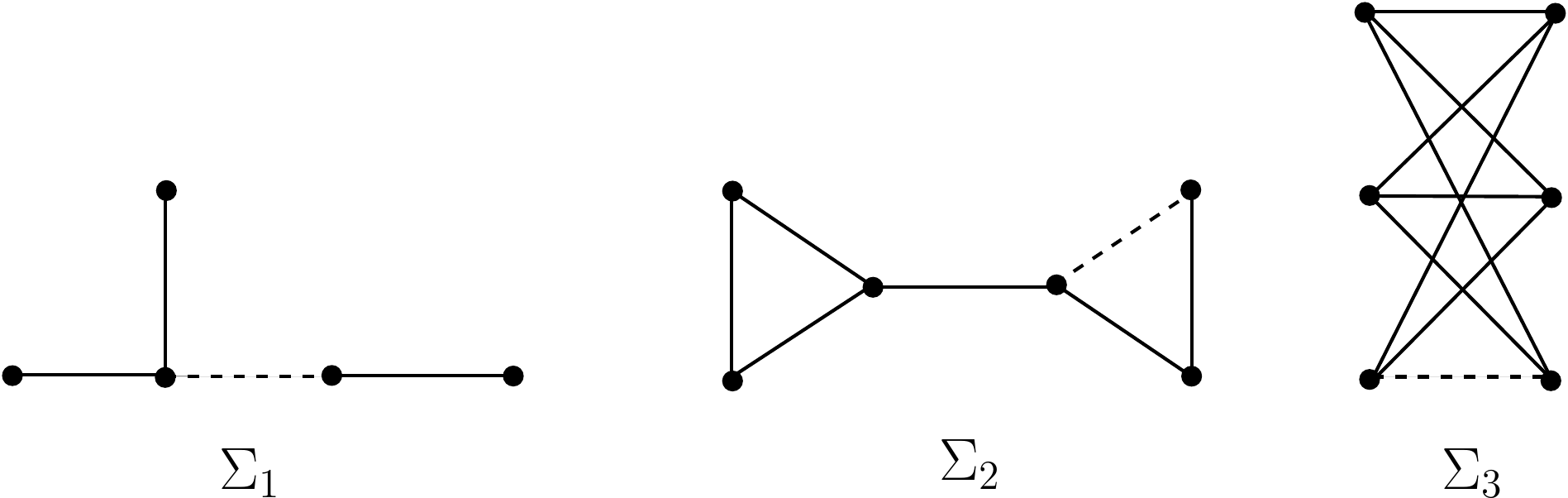}
    \caption{Three signed graphs $\Sigma_1$, $\Sigma_2$ and $\Sigma_3$.  Dashed lines represent negative ($-1$) edges, and solid lines represent positive ($+1$) edges.}\label{ex1}
\end{figure}

Table 1 is separated into three smaller tables.  The top table gives the various bounds for the Laplacian spectral radius of the signed graphs in Figure \ref{ex1}.  The table in the middle gives the various bounds for the Laplacian spectral radius of the unsigned graphs $\Gamma_i$  (the underlying graphs in Figure \ref{ex1}).  The bottom table gives the various bounds for the signless Laplacian spectral radius of the unsigned graphs $\Gamma_i$.  The columns are aligned so related bounds can be compared.  Specifically, each of the bounds in a distinct column are a consequence of the bounds listed in the top table which are \eqref{NeqLB1}, \eqref{NeqLB2}, \eqref{NeqLB3}, \eqref{UB1}, \eqref{UB2}, \eqref{NewUBSG}, \eqref{TrI1}, \eqref{TrI2} and \eqref{TrI3}.

Readers more familiar with signed graph theory might wonder why $\Sigma_1$ has a negative edge even though it is switching equivalent to the all positive tree.  This was done in order to compare the bounds which depend on the edge signs.
\begin{center}
\begin{table}[h!]
\caption{Values of the various bounds for the signed graphs $\Sigma_1$, $\Sigma_2$ and $\Sigma_3$.}
    \begin{tabular}{ c| c| c| c | c | c | c | c | c | c | c |}
    \cline{2-11}
      & $\lambda_{\max}(L(\Sigma_i))$ &\eqref{NeqLB1} &\eqref{NeqLB2}&\eqref{NeqLB3}& \eqref{UB1} &\eqref{UB2} & \eqref{NewUBSG} & \eqref{TrI1} &  \eqref{TrI2}  &  \eqref{TrI3} \\
    \hline
    \multicolumn{1}{|c|}{$\Sigma_1$} & 4.170 & 0.8 & 1.265 & 1.776 & 4.449 & 4.645  &  4.449 &  1.826 &  1.651  &  2.067\\
    \hline
     \multicolumn{1}{|c|}{$\Sigma_2$} & 4.842 & 0.667 &1.155  &1.671& 5.266 & 5.464 &  5.908 &  2.221 &  2.095  &  2.527\\
    \hline
    \multicolumn{1}{|c|}{$\Sigma_3$} & 5.561 & 0.667 & 1.155 & 1.747 &6 & 6 &  6.873 &  2.898 &  2.785  &  3.188\\
    \hline
    \multicolumn{7}{c}{} \\ \cline{2-11}    
     & $\lambda_{\max}(L(\Gamma_i,+1))$ &\eqref{NeqLB1}&\eqref{NeqLB2}&\eqref{NeqLB3}& \eqref{UB1} &\eqref{UB2} & \eqref{NewUBL} & \eqref{TrIL1} &  \eqref{TrIL2}  &  \eqref{TrIL3} \\
    \hline
    \multicolumn{1}{|c|}{$\Sigma_1$} & 4.170 &0&0&0&4.449 & 4.645 &  4.449 &  1.826 &  1.651  &  2.067\\
    \hline
    \multicolumn{1}{|c|}{$\Sigma_2$} & 4.562 &0&0&0& 5.266& 5.464  &  5.453 &  2.720 &  2.621  &  2.916\\
    \hline
     \multicolumn{1}{|c|}{$\Sigma_3$} & 6 &0&0&0& 6 & 6& 6  & 3.550 &  3.509  &  3.649\\
    \hline
    \multicolumn{7}{c}{} \\ \cline{2-11}   
    & $\lambda_{\max}(L(\Gamma_i,-1))$ &\eqref{NeqSLB1}&\eqref{NeqSLB2}&\eqref{NeqSLB3}&\eqref{UB1} &\eqref{UB2} & \eqref{NewUBSL} & \eqref{TrISL1} &  \eqref{TrISL2}  &  \eqref{TrISL3} \\
    \hline
    \multicolumn{1}{|c|}{$\Sigma_1$} & 4.170 & 3.2 & 3.578& 3.752&4.449 & 4.645 &  4.449 &  1.826 &  1.651  &  2.067\\
    \hline
    \multicolumn{1}{|c|}{$\Sigma_2$}  & 5 &4.667&4.761&4.820& 5.266& 5.464 &  5.908 &  2.221 &  2.110  &  2.506\\
    \hline
    \multicolumn{1}{|c|}{$\Sigma_3$}  & 6 & 6 & 6 & 6 & 6& 6 &  6 &  3.550 &  3.509  &  3.649\\
    \hline
    \end{tabular}
    \end{table}
\end{center}

Li and Li \cite{MR2477235} list the known bounds for the Laplacian spectral radius of a signed graph.  The following is a list of the stronger upper bounds mentioned in their paper:

\begin{align}
\lambda_{\max}(L(\Sigma)) &\leq \max_{e_{ij}\in E(\Sigma)} \frac{d_i(d_i+m_i)+d_j(d_j+m_j)}{d_i+d_j}, \label{kb1}\\
\lambda_{\max}(L(\Sigma)) &\leq \max_{e_{ij}\in E(\Sigma)} \Big\{ 2+\sqrt{d_i(d_i+m_i-4)+d_j(d_j+m_j-4)+4} \Big\},\label{kb2}\\
\lambda_{\max}(L(\Sigma)) &\leq \max_{v_i\in V(\Sigma)} \big\{ d_i+\sqrt{d_im_i} \big\}, \label{kb3}\\
\lambda_{\max}(L(\Sigma)) &\leq \max_{e_{ij}\in E(\Sigma)} \frac{d_i+d_j+\sqrt{(d_i-d_j)^2+4\sqrt{d_id_jm_im_j}}}{2}. \label{kb4}
\end{align}

Hou, Li and Pan \cite{MR1950410} generalized the following lower bound to signed graphs which was previously known for unsigned graphs: 
\begin{equation}\label{kb5}
\lambda_{\max}(L(\Sigma)) \geq  \Delta +1.
\end{equation}

Table 2 compares these previously known bounds to the signed graphs in Figure \ref{ex1}.  Notice that none of these bounds depend on edge signs unlike the new bounds from Section \ref{MainR}.  Therefore, we do not need to consider the Laplacian and signless Laplacian of the unsigned graphs $\Gamma_i$ separately.

\begin{center}
\begin{table}[h!]
\caption{Values of previously known bounds for the signed graphs $\Sigma_1$, $\Sigma_2$ and $\Sigma_3$.}
    \begin{tabular}{ c| c| c| c | c | c |}
    \cline{2-6}
      & \eqref{kb1} &\eqref{kb2}&\eqref{kb3}& \eqref{kb4} &\eqref{kb5}  \\
    \hline
    \multicolumn{1}{|c|}{$\Sigma_1$} & 4.250 & 4.236 & 5 & 4.562 & 4\\
    \hline
     \multicolumn{1}{|c|}{$\Sigma_2$} & 5.333 & 5.464 & 5.646 & 5.646 & 4 \\
    \hline
    \multicolumn{1}{|c|}{$\Sigma_3$}  & 6 & 6 & 6 & 6 & 4  \\
    \hline
    \end{tabular}
    \end{table}
\end{center}


\begin{thebibliography}{19}
\bibitem{MR1130611} Richard A. Brualdi and Herbert J. Ryser, {\bf Combinatorial Matrix Theory}, Encyclopedia of Mathematics and its Applications, Cambridge University Press, 1991. MR 93a:05087.

\bibitem{Chaiken} Seth Chaiken, A combinatorial proof of the all minors matrix tree theorem. \emph{SIAM J. Algebraic Discrete Methods} {\bf 3} (1982), 319–329. MR 83h:05062. Zbl. 495.05018.

\bibitem{MR2312332} Drago{\v{s}} Cvetkovi{\'c}, Peter Rowlinson and Slobodan K. Simi{\'c}, Signless Laplacians of finite graphs, \emph{Linear Algebra Appl.}, {\bf 423} (2007), 155--171. MR 2008c:05105.

\bibitem{MR2401311} Drago{\v{s}} Cvetkovi{\'c}, Peter Rowlinson and Slobodan K. Simi{\'c}, Eigenvalue bounds for the signless Laplacian, \emph{Publ. Inst. Math. (Beograd) (N.S.)}, {\bf 81(95)} (2007), 11--27. MR 2009e:05181.

\bibitem{MR1084815} Roger A. Horn, Charles R. Johnson, {\bf Matrix Analysis}, Cambridge University Press, (1990), MR 91i:15001.

\bibitem{MR2350694} Ting-Zhu Huang and Lin Wang, Improving bounds for eigenvalues of complex matrices using traces, \emph{Linear Algebra Appl.},
{\bf 426} (2007), 841--854. MR 2008g:15031.

\bibitem{MR2156977} Yao Ping Hou, Bounds for the least Laplacian eigenvalue of a signed graph, \emph{Acta Mathematica Sinica (English Series)}, {\bf 21} (2005), 955--960. MR 2006d:05120.

\bibitem{MR1950410} Yaoping Hou, Jiongsheng Li, Yongliang Pan, On the Laplacian eigenvalues of signed graphs, \emph{Linear Multilinear Algebra}, {\bf 51}, (2003), 21--30. MR 2003j:05084.  

\bibitem{MR2477235} Hong-Hai Li, Jiong-Sheng Li, An upper bound on the Laplacian spectral radius of the signed graphs, \emph{Discuss. Math. Graph Theory} {\bf 28} (2008), 345--359, MR 2010a:05115.

\bibitem{MR2390486} Tian-fei Wang, Several sharp upper bounds for the largest Laplacian eigenvalue of a graph, \emph{Sci. China Ser. A}, {\bf 50} (2007), 1755--1764. MR 2009a:05131.

\bibitem{MR676405} Thomas Zaslavsky, Signed graphs, \emph{Discrete Appl. Math.} {\bf 4}, (1982), 47--74. MR 84e:05095a. Erratum, \emph{ibid.}, {\bf 5} (1983), 248. MR 84e:05095b.


\end{thebibliography}
\end{document}